\documentclass[12pt]{amsart}
\usepackage{fancyhdr,graphicx}
\usepackage{amsfonts}
\usepackage{amssymb}
\usepackage{amsthm}
\usepackage{newlfont}
\usepackage{amsmath} %%%%%%%%%
\usepackage[frame,arrow,curve,matrix]{xy}

\newtheorem{thm}{Theorem}[section]
\newtheorem{Lemma}[thm]{Lemma}
\newtheorem{cor}[thm]{Corollary}
\newtheorem{pro}[thm]{Proposition}
\newtheorem{definition}[thm]{Definition}

\newtheorem{remark}[thm]{Remark}
\newtheorem{conj}{Conjecture}[section]
\newtheorem{exam}{Example}[section]
\usepackage{mathrsfs}  %出现数学希腊符号$\mathscr{OS}$
\usepackage{subfig}
\usepackage{picinpar}
\begin{document}
\title[Cut points of virtual links and the arrow polynomial of twisted links]{One conjecture on cut points of virtual links and the arrow polynomial of twisted links}

\author{Qingying Deng}
\address{School of Mathematics and Computational Science, Xiangtan University, Xiangtan, Hunan 411105,
P. R. China}
\email {qingying@xtu.edu.cn}

\begin{abstract}
Checkerboard framings are an extension of checkerboard colorings for virtual links.
According to checkerboard framings, in 2017, Dye obtained an independent invariant of virtual links: the cut point number. Checkerboard framings and cut points can be used as a tool to extend
other classical invariants to virtual links.
We prove that one of the conjectures in Dye's paper is correct.
Moreover, we analyze the connection and difference between checkerboard framing obtained from virtual link diagram by adapting cut points and twisted link diagram obtained from virtual link diagram by introducing bars.
By adjusting the normalized arrow polynomial of virtual links, we generalize it to twisted links.
And we show that it is an invariant for twisted link.
Finally, we figure out three characteristics of the normalized arrow polynomial of a checkerboard colorable twisted link, which is a tool of detecting checkerboard colorability of a twisted link.
The latter two characteristics are the same as in the case of checkerboard colorable virtual link diagram.
\end{abstract}
\maketitle

$\mathbf{Keywords:}$ Twisted link; virtual link; checkerboard framings; cut points; arrow polynomial; checkerboard colorability.

\vskip0.5cm

\section{Introduction}
Virtual knot theory is a generalization of knot theory, one motivation is based on Gauss code \cite{VKT}.
Virtual links correspond to stable equivalence classes of links in oriented 3-manifolds which are thickened closed oriented surfaces \cite{VKT,Kam,Carter}.
Twisted knot theory, introduced by M.O. Bourgoin in \cite{Bourgoin}, is an extension of virtual knot theory, which is focused on link diagrams on closed, possibly non-orientable surfaces.
Twisted links are in one-to-one correspondence with abstract links on (oriented or non-oriented) surfaces \cite{Bourgoin}, and in one-to-one correspondence with stable equivalence classes of links in oriented thickenings of closed surfaces \cite{Bourgoin}.
Virtual links are regarded as twisted links. Recently, S. Kamada and N. Kamada discussed when two virtual links are equivalent as twisted links, and give a necessary and sufficient
condition for this to be the case in \cite{Kamada2020}.

A \emph{virtual link diagram} is a link diagram which may have virtual crossings, which are encircled crossings without over-under
information. A \emph{twisted link diagram} is a virtual link diagram which may have some bars. Three examples of
twisted link diagrams are depicted in Fig. \ref{F:twisted knot}, the latter one is virtual link diagram. (In \cite{Kamada2010}, a twisted link diagram is defined to be a virtual link diagram possibly
with some 2-valent vertices.)

A \emph{virtual link} is an equivalence class of a virtual link diagram by Reidemeister moves and virtual Reidemeister moves in Fig. \ref{F:classicalRM} and Fig. \ref{RM}.
Note that the virtual Reidemeister moves are equivalent to detour move (See Fig. \ref{dm}).
A \emph{twisted link} is an equivalence class of a twisted link diagram by Reidemeister moves, virtual Reidemeister moves
and \emph{twisted Reidemeister moves} in Fig. \ref{F:classicalRM}, \ref{RM} and \ref{F:T123}.
All of these are called \emph{extended Reidemeister moves}.
Note that we can use the virtual Reidemeister moves or detour move even though there are bars on some arcs based on move T1 as shown in Fig. \ref{Fig.detourfortl}
\begin{figure}[!htbp]
  \centering
  \includegraphics[width=8cm]{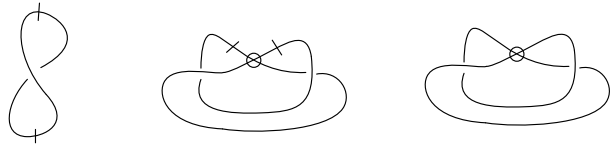}
   \renewcommand{\figurename}{Fig.}
  \caption{A onefoil twisted knot (left), a nonorientable twisted knot (middle), a virtual knot (right).}
  \label{F:twisted knot}
\end{figure}
\begin{figure}[!htbp]
  \centering
  \includegraphics[width=4cm]{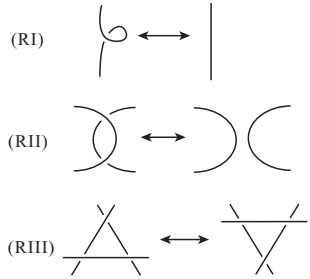}
   \renewcommand{\figurename}{Fig.}
  \caption{Reidemeister moves.}
  \label{F:classicalRM}
\end{figure}

\begin{figure}[!htbp]
  \centering
  \includegraphics[width=6cm]{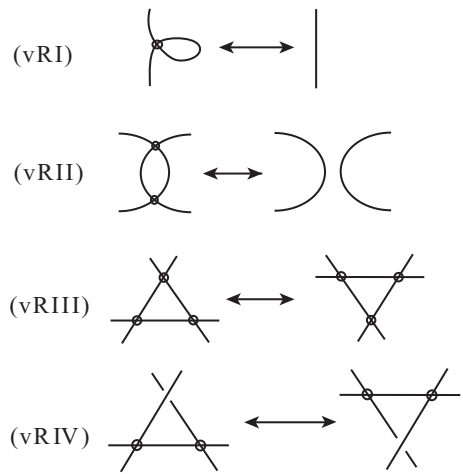}
   \renewcommand{\figurename}{Fig.}
  \caption{Virtual Reidemeister moves.}
  \label{RM}
\end{figure}

\begin{figure}[!htbp]
  \centering
  \includegraphics[width=8cm]{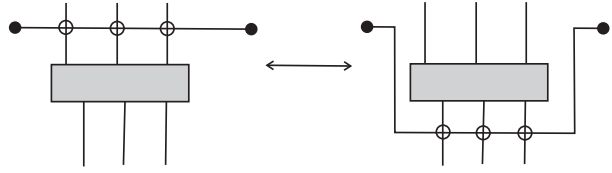}
   \renewcommand{\figurename}{Fig.}
  \caption{Detour move.}
  \label{dm}
\end{figure}
\begin{figure}[!htbp]
  \centering
  \includegraphics[width=8cm]{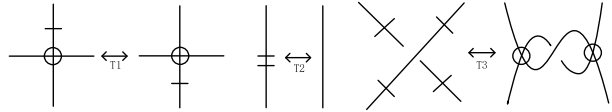}
   \renewcommand{\figurename}{Fig.}
  \caption{Twisted Reidemeister moves.}
  \label{F:T123}
\end{figure}
\begin{figure}[!htbp]
  \centering
  \includegraphics[width=8cm]{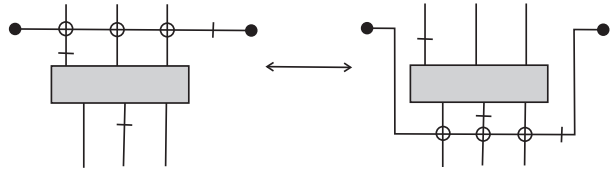}
    \renewcommand{\figurename}{Fig.}
  \caption{Detour move with one bar.}
  \label{Fig.detourfortl}
\end{figure}

 Checkerboard framings are an extension of checkerboard colorings for virtual links.
From checkerboard framings, Dye obtained an independent invariant of a virtual link: the
cut point number $\mathcal{P}(K)$ of a virtual link $K$ in \cite{Dye}. Checkerboard framings and cut points can be used as a tool to extend other classical invariants to virtual links.
In paper \cite{DyeKaeKau}, Dye used cut points or cut loci to extend the definition
of integer Khovanov homology \cite{Bar-Natan02} to virtual link diagrams. In \cite{DyeKaeKau}, cut points
were also used to extend the Rasmussen-Lee invariant (originally defined in \cite{Bar-NatanMorrison06})
to virtual links.
In \cite{Dye}, Dye made the following Conjecture \ref{conj1} about any two checkerboard framings of a
virtual link diagram. In this paper, we mainly give out a proof of Conjecture \ref{conj1} that is ture in Section \ref{proveconj}.
Moreover, we analyze the connection and difference between checkerboard framing obtained from virtual link diagram by adapting cut points and twisted link diagram obtained from virtual link diagram by introducing bars.
\begin{conj}\label{conj1}
Any two checkerboard framings of a diagram are related by a
sequence of the cut point moves I and II (Fig. \ref{F:cutpointmove}).
\end{conj}
\begin{figure}[!htbp]
  \centering
  \includegraphics[width=8cm]{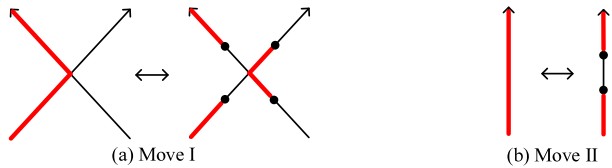}
   \renewcommand{\figurename}{Fig.}
  \caption{Cut point moves.}
  \label{F:cutpointmove}
\end{figure}

Dye and Kauffman introduced the arrow polynomial which is an invariant of oriented virtual knots and links in \cite{DKArrow}. This invariant
takes values in the ring $\mathbb{Z}[A,A^{-1},K_{1},K_{2},\cdots]$ where the $\{K_{i}|i\in\mathbb{N}\}$ are an infinite set of
independent commuting variables that also commute with the Laurent polynomial
variable $A$. This invariant was independently constructed by
Miyazawa in \cite{Miyazawa} using a different definition.
%It seems difficult to extend the original definition of arrow polynomial
%of a virtual link to an invariant of a twisted link.
In this paper, we extend the polynomial invariants defined by Dye and Kauffman to invariants of twisted links in Section \ref{exarrowpoly}.
Moreover, we figure out three characteristics of the normalized arrow polynomial of a checkerboard colorable twisted link, which is a tool of detecting checkerboard colorability of a twisted link.
The latter two characteristics are the same as in the case of checkerboard colorable virtual link diagram.

\section{Checkerboard framings of virtual link diagrams and checkerboard colorability of twisted links}\label{proveconj}
\subsection{Checkerboard framings of virtual link diagrams}
For a classical link diagram $D$, the diagram can be checkerboard colorable since its dual graph of the underlying 4-valent graph of $D$ is bipartite. The planar regions are alternately colored
so that the unbounded region of the plane is white.
The notion of a checkerboard coloring for a virtual link diagram was first introduced
by Kamada \cite{Kamada,Kamada2004} by using corresponding abstract link diagram defined in \cite{Kam}.
A virtual link diagram is said to be $\textit{checkerboard colorable}$ if there is a coloring of a small neighbourhood of one side of each arc in the diagram such that near a classical crossing the coloring alternates, and near a virtual crossing the colorings go through independent of the crossing strand and its coloring.
A virtual link is said to be $\textit{checkerboard colorable}$ if it has a checkerboard colorable diagram.
Two examples are given in Fig. \ref{ckb}.

\begin{figure}[!htbp]
  \centering
  \includegraphics[width=3.5in]{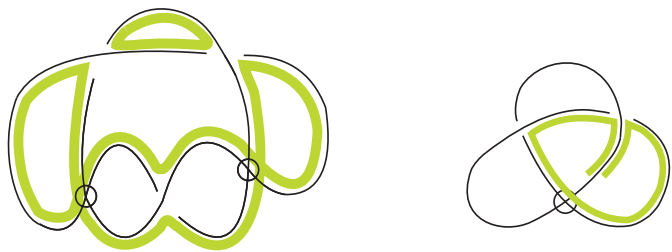}
  \renewcommand{\figurename}{Fig.}
\caption{The left virtual link diagram is checkerboard colorable (to draw only one color, omit the other), the right is not checkerboard colorable.}\label{ckb}
\end{figure}

For virtual link diagram $D$, let $|D|$ be the \emph{underlying 4-valent graph} which is obtained from $D$
by regarding all classical crossings as the vertices of $|D|$ and keeping or ignoring virtual crossings.
Then the segment between two classical crossings is called as an \emph{edge} of $|D|$.
In \cite{Kamada02} it is observed that giving a checkerboard coloring for $D$ is equivalent to giving
an \emph{alternate orientation} to $|D|$ that is an assignment of orientations to the edges
of $|D|$ satisfying the conditions illustrated in Fig. \ref{Fig.alterorietation}.
Note that every checkerboard colorable virtual link diagram has two kinds of checkerboard colorings and not every virtual link diagram is checkerboard colorable.
Checkerboard colorability of a virtual link diagram is not necessarily preserved by generalized Reidemeister moves.
See the explanation in \cite{Ima}.
Furthermore, note that an alternating virtual link diagram \cite{Kamada2004} must be checkerboard colorable (For example,
we can obtain an alternation orientation if we orient under crossing with two ``in'' and over crossing with two``out''.), on the contrary, it is not true.

\begin{figure}[!htbp]
  \centering
  \includegraphics[width=3.5cm]{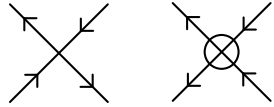}
  \renewcommand{\figurename}{Fig.}
\caption{An alternate orientation.}\label{Fig.alterorietation}
\end{figure}

In \cite{Dye}, Dye used thick edges and thin edges to distinguish ``colors'' on the edges.
The edge coloring version of a checkerboard coloring satisfies two conditions.
\begin{enumerate}
  \item[(1)] A color is assigned to each edge.
  \item[(2)] The color assignments respect crossings: the left-hand and right-hand side of crossings
have distinct color assignments as shown in Fig. \ref{Fig.colorassignment}.
\end{enumerate}
\begin{figure}[!htbp]
  \centering
  \includegraphics[width=2.5cm]{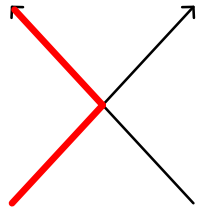}
  \renewcommand{\figurename}{Fig.}
\caption{Color assignment at a classical crossing.}\label{Fig.colorassignment}
\end{figure}

The second condition implies that the edge colors alternate as the orientation
of a component is followed.

In \cite{Dye}, Dye defined a checkerboard framing of a virtual link
diagram $D$. Edges of the virtual link diagram are bounded by classical crossings.
That is, the edges of the virtual link diagram correspond to the edges in its underlying 4-valent graph $|D|$.
A \emph{checkerboard framing} is an assignment of cut points to the edges of a virtual link
diagram (at most one cut point on each edge) and colors to the resulting set of edges. A cut point subdivides an edge into
two edges. Edges of the checkerboard framed diagram are bounded by two classical
crossings or a cut point and a classical crossing. A checkerboard framing satisfies
three following conditions.
%checkerboard framing:加入割点并正常染色边；
\begin{enumerate}
  \item Each edge is assigned a color.
  \item Edge colors alternate.
  \item Edge colors respect crossings (see Fig. \ref{Fig.colorassignment}).
\end{enumerate}

For a checkerboard framing of $D$, we note that for each cut point, if you see it as a 2-valent vertex and each classical crossing as a 4-valent vertex, then its incident edges have two colors.

A \textit{modified checkerboard framing} of a virtual link diagram may contain more
than one cut point on an edge. However, the edge colors respect crossings.

For a virtual link diagram $D$, a checkerboard framing $F$ of $D$ is denoted as
$(D, F)$. The cut point number of $(D, F)$ is the number of cut points in $(D, F)$ and
is denoted as $\mathcal{P}((D, F))$.

\begin{pro}(\cite{Dye})
For all checkerboard framed, virtual link diagrams $(D, F)$, we
have the following statements.
\begin{itemize}
 \item $\mathcal{P}((D, F))$ is even.
  \item If $D$ is a diagram with $n$ crossings, then $0\le \mathcal{P}((D, F)) \le 2n$.
\end{itemize}
\end{pro}

The minimum number of cut points required by a virtual link diagram is

\begin{equation}\label{Dcutpoint}
\mathcal{P}_{d}(D)=Min\{\mathcal{P}((D,F))|F~is~ a~ framing~ of~ D\}.
\end{equation}

Then, $\mathcal{P}(K)$ is the minimum number of cut points required to frame any virtual
link diagram $D$ equivalent to $K$. That is, the cut point number of $K$ is

\begin{equation}\label{Kcutpoint}
\mathcal{P}(K)=Min\{\mathcal{P}((D,F))|D\sim K~ and~ F~ is~ a~ framing~ of~ D\}.
\end{equation}

\begin{thm}(\cite{Dye})
$\mathcal{P}(K)$ is a virtual link invariant.
\end{thm}

\begin{cor}(\cite{Dye})
For all virtual links $K$, if $\mathcal{P}(K)>0$, then $K$ is not a classical link.
\end{cor}

\begin{thm}(\cite{Dye})
For all virtual link diagrams $D$, let $v_{d}(D)$ denote the number of
virtual crossings in the diagram. Then
$\mathcal{P}_{d}(D)\le 2v_{d}(D)$.
\end{thm}

\begin{cor}(\cite{Dye})
For all virtual links $K$, let $v(K)$ denote the minimum number of
virtual crossings in any diagram of $K$, then
$\mathcal{P}(K)\le 2v(K)$.
\end{cor}

A checkerboard framing of a virtual link diagram $D$ can be modified using the cut points moves in Fig. \ref{F:cutpointmove}. Note that Move I can add four cut points to four edges incident to a classical crossing or delete.
The number of cut points on an edge can be reduced to one or zero using Move II resulting in a checkerboard framing.

First we give out a proof of Conjecture \ref{conj1} as follows.
It is obvious that the cut point move I and II does not change the property of checkerboard framing.

\begin{proof}
Let the underlying graph of a virtual link diagram $D$ with $n$ classical crossings be $|D|$ with $n$ 4-valent vertices. Since every edge of $|D|$ will have zero or one cut point, then we can get a set $2^{E(|D|)}$ of $2n$ dimension $0,1$ array modular 2, there are $2^{2n}$ arrays. Denote the sets of all non-checkerboard framings and checkerboard framings by $NCF$ and $CF$, respectively.
For any $f,g\in CF$, we assume that the checkerboard colorings are $\mathcal{C}_{f}$ and $\mathcal{C}_{g}$ corresponding to $f$ and $g$, respectively.
Note that for any $f_{0}\in CF$, there are two checkerboard colorings corresponding to it, we just arbitrarily choose one.
It is obvious that for each classical crossing $c$ of $D$, there are colorings $\mathcal{C}_{f}(c)$ and $\mathcal{C}_{g}(c)$ around its adjacent edges in $\mathcal{C}_{f}$ and $\mathcal{C}_{g}$, respectively.
Then colorings $\mathcal{C}_{f}(c)$ and $\mathcal{C}_{g}(c)$ are same or different.
See one example as illustrated in Fig. \ref{fgexample}.
For checkerboard framing $f$, if we use first move I for all classical crossings $c's$ in where $\mathcal{C}_{f}(c)$ and $\mathcal{C}_{g}(c)$ are different as shown in Fig. \ref{changecolor}, then we can obtain one new checkerboard framing $f'$ and checkerboard coloring $\mathcal{C}_{f'}$.
Note that $\mathcal{C}_{f'}$ and $\mathcal{C}_{g}$ only are different in some edges whose two endpoints are cut-points.
Next we use move II in $f'$ to reduce the cut-points such that the number of cut-point is 0 or 1 in each edge of $D$, in the same time, the checkerboard coloring $\mathcal{C}_{f'}$ transforms to $\mathcal{C}_{g}$, that is, we obtain the checkerboard framing $g$.

\begin{figure}[!htbp]
  \centering
  \includegraphics[width=12cm]{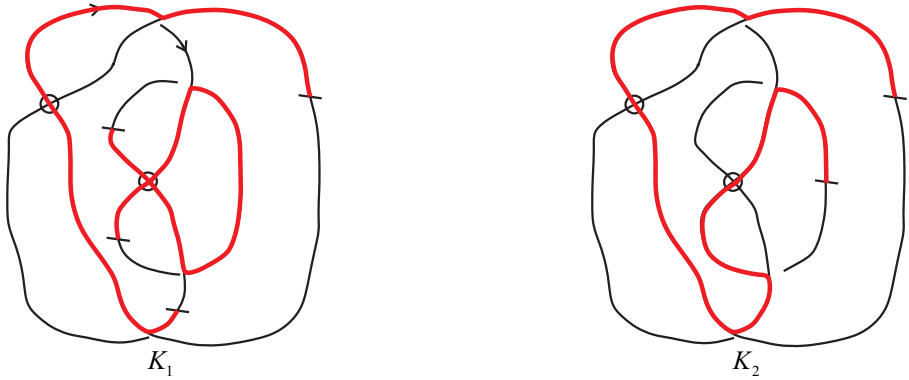}
   \renewcommand{\figurename}{Fig.}
  \caption{Checkerboard framings $f$ and $g$ and their checkerboarding colorings.}
  \label{fgexample}
\end{figure}

\begin{figure}[!htbp]
  \centering
  \includegraphics[width=12cm]{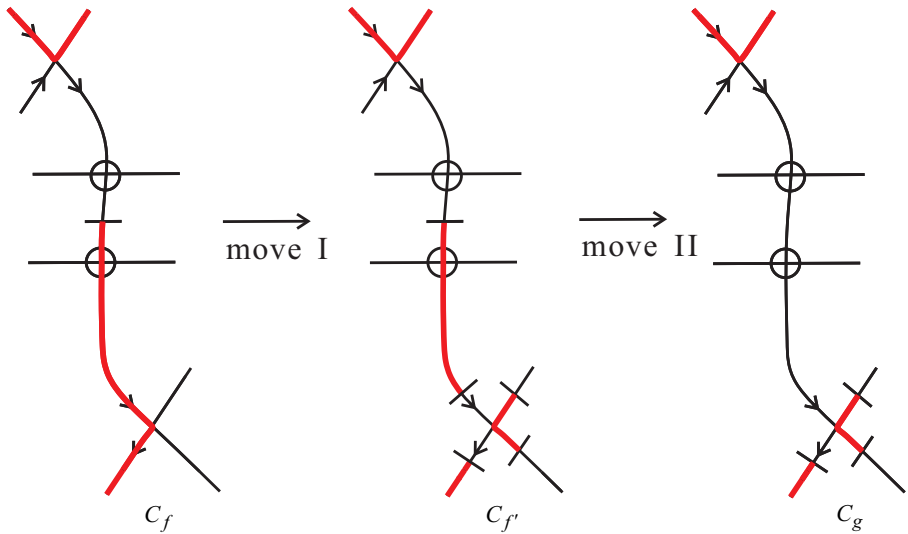}
   \renewcommand{\figurename}{Fig.}
  \caption{Switch the color of a classical crossing.}
  \label{changecolor}
\end{figure}
\end{proof}

\begin{thm}\label{conjtoTh1}
Any two checkerboard framings of a diagram are related by a
sequence of the cut point moves I and II (Fig. \ref{F:cutpointmove}).
\end{thm}

\subsection{The checkerboard colorability of twisted links}
The faces of a twisted link diagram are closed
curves that run along the immersed curve and have the relationship with the crossings,
virtual crossings, and bars as show in Fig. \ref{facesofdiagram}. At a crossing, a face turns so as to
avoid crossing the link diagram. At a virtual crossing, a face goes through the virtual
crossing. At a bar, a face crosses to the other side of the link diagram.
A twisted link
diagram is \emph{checkerboard colorable}(or \emph{two-colorable} in \cite{Bourgoin}) if its faces can be assigned one of two colors such that the
arcs of the link diagram between two crossings always separate faces of one color from
those of the other as show in Fig. \ref{Colorofcrossingbar}.
\begin{figure}[!htbp]
  \centering
  \includegraphics[width=8cm]{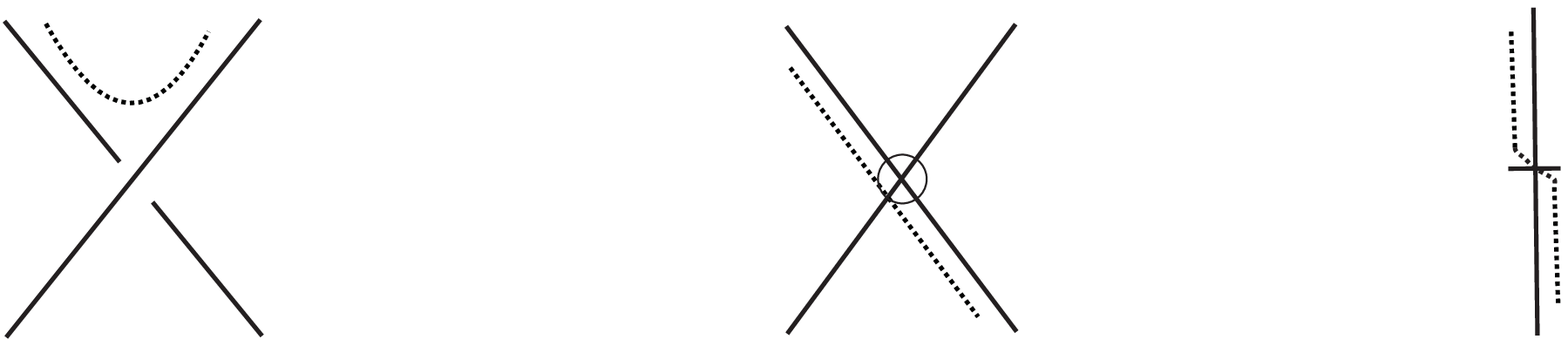}
   \renewcommand{\figurename}{Fig.}
  \caption{Switch the color of a classical crossing.}
  \label{facesofdiagram}
\end{figure}
\begin{figure}[!htbp]
  \centering
  \includegraphics[width=8cm]{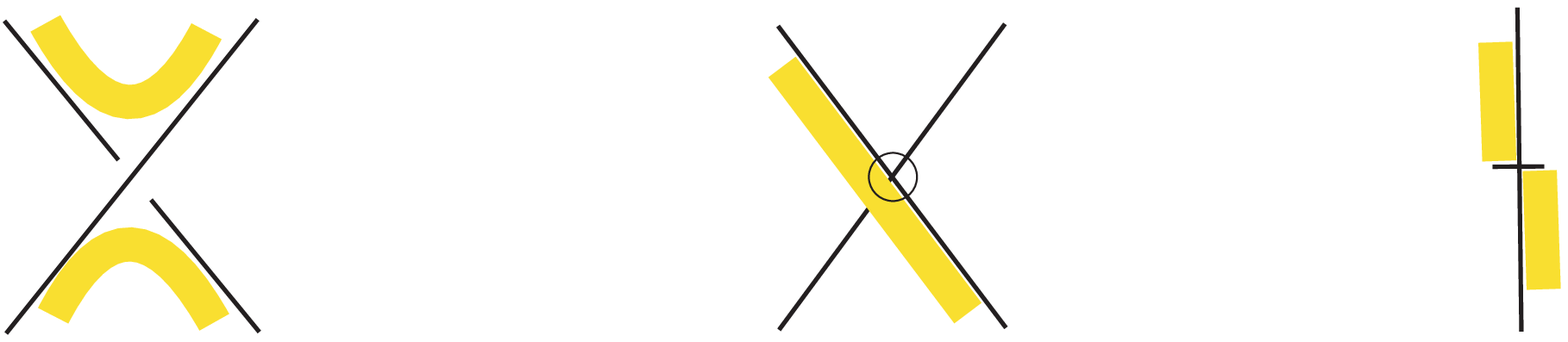}
   \renewcommand{\figurename}{Fig.}
  \caption{Switch the color of a classical crossing.}
  \label{Colorofcrossingbar}
\end{figure}

Note that the three new twisted Reidemeister moves do not change the checkerboard colorability of a twisted link diagram.
But the extended Reidemeister moves may change the checkerboard colorability of a checkerboard colorable twisted link diagram.

Let $D$ be a twisted link diagram, we denote the number of bars on $D$ by $b(D)$.
Then, $b(L)$ is the minimum number of bars contained in any twisted
link diagram $D$ equivalent to $L$. That is, the \emph{bar number} of $L$ is

\begin{equation}\label{barnumber}
b(L)=Min\{b(D)|D\sim L\}.
\end{equation}

\begin{thm}
$b(L)$ is a twisted link invariant.
\end{thm}

\begin{Lemma}\label{Lemma:pDeven}
Let $D$ be a checkerboard colorable twisted link diagram with $n$ classical crossings, then the number $b(D)$ of bars on $D$ is even and $0\leq b(D)\leq 2n$ after using move T1 and T2.
\end{Lemma}
\begin{proof}
$D$ has $2n+b(D)$ edges if we view each classical crossing and each bar as a 4-valent vertex and 2-valent vertex, respectively.
If $D$ is a knot, the orientations of continuous two edges for a vertex alternate between ``in'' and ``out''.
For the first edge and the last edge to have different orientations, $D$ must have an
even number of edges, that is, even number $b(D)$ of bars is even.

If $D$ is a link diagram, it is possible for a component to share classical crossings with
other components. This means that the component may have an odd number of
edges.
If the edges of the odd component
have an alternate orientation, then the component has an odd number of bars. However,
components of this type occur in pairs so that the total number of bars is
even.

In the worst case scenario, every edge in its underlying diagram could have a bar after using move T1 and T2. If a
diagram has $n$ crossings, then there would be $2n$ bars. As a result, there is at most a bar on each edge and $0\leq b(D)\leq 2n$.

Since the underlying diagram of $D$ is a virtual link diagram, thus we can adapt the same formulation of Proposition 3.1 in \cite{Dye}.
We only formulate the same statement in different ways.
\end{proof}

\begin{cor}
For all twisted links $L$, if $b(L)>0$, then $L$ is not a virtual link.
\end{cor}

\begin{cor}
For all twisted links $L$, let $v(L)$ denote the minimum number of
virtual crossings in any diagram of $L$, then
$b(L)\le 2v(L)$.
\end{cor}

Note that if $D$ is a non-checkerboard colorable twisted link diagram, then $b(D)$ is odd or even.
We can place all bars near to virtual crossing such that the virtual crossings behave like classical crossings with regard to the checkerboard coloring.

\begin{cor}
Let $L$ be a checkerboard colorable twisted link.
For any diagrams $D$ of $L$, then $b(D)$ is even.
\end{cor}
\begin{proof}
$b(D)$ is even for checkerboard colorable diagram $D$ by Lemma \ref{Lemma:pDeven}.
Based on the extended Reidemeister moves do not change the parity of the number of bars, we can obtain the conclusion.
\end{proof}

Note that the converse of this corollary is not true since there are non-checkerboard colorable virtual link diagram with zero bar.

We shall conclude this section by figuring out the connection and difference between Dye's checkerboard framing and checkerboard coloring of twisted link diagram.
It is obvious that we can construct a bijection between Dye's checkerboard framing and checkerboard coloring of the corresponding twisted link diagram.
For a virtual link diagram $D$, let $F$ be a checkerboard framing of $D$. Then we can obtain a checkerboard colorable twisted link diagram $D'$ from $F$ by replacing each cut-point with a bar. Vice versa.
Dye's checkerboard framing is a combinatorial object, but checkerboard colorability of twisted link diagram is the property of the twisted link diagram.
Thus, if we just focus on the checkerboard colorability of a diagram, then the role of cut-points in virtual link diagram and the role of bars in a twisted link diagram are the same.
But bars in a twisted link diagram contain the topological construct of the diagram.

Let $D$ be a virtual link diagram.
Let $K_{1}$ and $K_{2}$ be two checkerboard colorable twisted link diagrams obtained by adding some bars to $D$.
Then $K_{1}$ and $K_{2}$ maybe inequivalent as twisted links.
See example of a pair of $K_{1}$ and $K_{2}$ in Fig \ref{fgexample}.
By computing their normalized arrow polynomials (see Section 3),
$\langle K_1 \rangle_{NA}=\frac{-\left(A^4+1\right) \text{K1}^2+A^4+A^2+1}{A^2}$, $\langle K_2 \rangle_{NA}=\frac{-\left(A^8+2 A^6+2 A^4+2 A^2+1\right) \text{K1}^2+\left(A^8+2 A^6+3 A^4+2 A^2+1\right)}{A^4}$, then they are not equivalent by Theorem \ref{invofNA}.
Thus we can obtain a lot inequivalent of twisted links from $D$ by adding bars on diagrams.

Theorem \ref{conjtoTh1} is true for checkerboard framings of a virtual link diagram, but
the above example of a pair of $K_{1}$ and $K_{2}$ illustrates that it is not true if the cut-point is replaced with bars for twisted link diagrams and the cut point moves I and II are replaced with move T1, T2 and T3, since the move T3 may change the topological construct of the diagram.

\section{Arrow polynomials of twisted links}\label{exarrowpoly}
In this section, we recall arrow polynomial of an oriented virtual link diagram defined by Dye and Kauffman \cite{DKArrow} which is a criterion of detecting checkerboard colorability of virtual links in \cite{DJK2020}.

Let $D$ be an oriented virtual link diagram. We shall denote the arrow polynomial
by the notation $\langle D\rangle_{A}$.
The \emph{arrow polynomial} of $D$ is based on the oriented state expansion as shown in Fig. \ref{Fig.arrow}.
Note that each classical crossing has two types of smoothing, one \emph{oriented smoothing} and other \emph{disoriented smoothing}.
Each disoriented smoothing gives rise to a \emph{cusp pair} where each \emph{cusp} is denoted by an angle with arrows either both entering the vertex or both leaving the vertex.
Furthermore, the angle locally divides the plane into two parts: One part is the span of an acute angle (of size less than $\pi$); the other part is the span of an obtuse angle.
The \emph{inside} of the cusp denotes the span of the acute angle.
It is obvious that the total number of cusps of a state circle after smoothing all classical crossing will be even according to the orientations on the edges of the state circle.
The structure with cusps is reduced according to a set of rules that ensures invariance of the state summation under generalized Reidemeister moves.
The basic conventions for this simplification are shown in Fig. \ref{Fig.reduction}.
When the insides of the cusps are on opposite sides of the connecting
segment (a ``zig-zag''), then no cancellation is allowed. Each state circle is seen as a \emph{circle graph} with extra nodes corresponding to the cusps. All graphs are taken up to virtual equivalence, as explained above. Fig. \ref{Fig.reduction} illustrates the simplification of three circle graphs. In one case the graph reduces to a circle with no vertices. In the other case there is no further cancellation, but the graph is equivalent to one without a virtual crossing.

Note that the virtual crossings in a state have the potential to effect the total number of cusps.
Use the reduction rule of Fig. \ref{Fig.reduction} so that each state
is a disjoint union of \emph{reduced circle graphs}.
Since such graphs are planar, each is equivalent to
an embedded graph (no virtual crossings) via the detour move, and the reduced forms of such
graphs have $2n$ vertices that alternate in type around the circle so that $n$ are pointing inward (the angle of the cusp is obtuse in the inside of the circle) and
$n$ are pointing outward (the angle of the cusp is acute in the inside of the circle). The circle with no vertices is evaluated as $d=-A^{2}-A^{-2}$ as is usual for these expansions, and the circle is removed from the graphical expansion.
Let $K_{n}$ denote the
circle graph with $2n$ alternating vertex types as shown in Fig. \ref{Fig.reduction} for $n=1$ and $n=2$. Each
circle graph contributes $d=-A^{2}-A^{-2}$ to the state sum and the graphs $K_{n}$ for $n\geq1$ remain in the graphical expansion. Each $K_{n}$ is an extra variable in the polynomial. Thus a product of the
$K_{n}$'s corresponds to a state that is a disjoint union of copies of these circle graphs.
Note that we continue to use the caveat that an isolated circle or circle graph (i.e. a state consisting in a single circle or single circle graph) is assigned a state circle value of
unity in the state sum. This assures that $\langle D\rangle_{A}$ is normalized so that the
unknot receives the value one.
Note that the arrow polynomial will reduce to the classical bracket polynomial
when each of the new variables $K_{n}$ is set equal to unity.

\begin{remark}\label{importantstatement}
Formally, we have a state summation for the arrow polynomial of a twisted link diagram analogously to that of the virtual link diagram.
For a state of $D$, there maybe exist a circle graph $G$ with cusps and bars.
For a twisted link diagram $D$, we keep above rules and add some new rules for a state.
If a reduced circle graph contains a cusp with two bars near it, we will reduce it as shown in Fig. \ref{Fig.cuspbar}.
Generally, for a circle graph $G$ with cusps $1,{2},\cdots,{2n}$ and several bars, if we consider a pair of bars, that is, one bar $A$ (resp. $B$) between cusps $i$ and $i+1$ (resp. $j-1$ and $j$), we assume that $i<j$ and $j-i\leq n$ without loss generality.
Then we can insert a pair of bars between cusps $i+k$ and $i+k+1$ ($1\leq k\leq j-i-2$), use the reduction rule of a cusp with two bars as shown in Fig. \ref{Fig.cuspbar}.
That is, we will obtain a new circle graph $G'$ with new cusps information but without bars.
Then we will reduce $G'$ to $K_{1},K_{2},...$ if the number of bars in $G$ is even or a trivial loop with a bar if the number of bars in $G$ is odd (as shown in Fig. \ref{Fig.K1bar}) .
\end{remark}
\begin{figure}[!htbp]
  \centering
  \includegraphics[width=6cm]{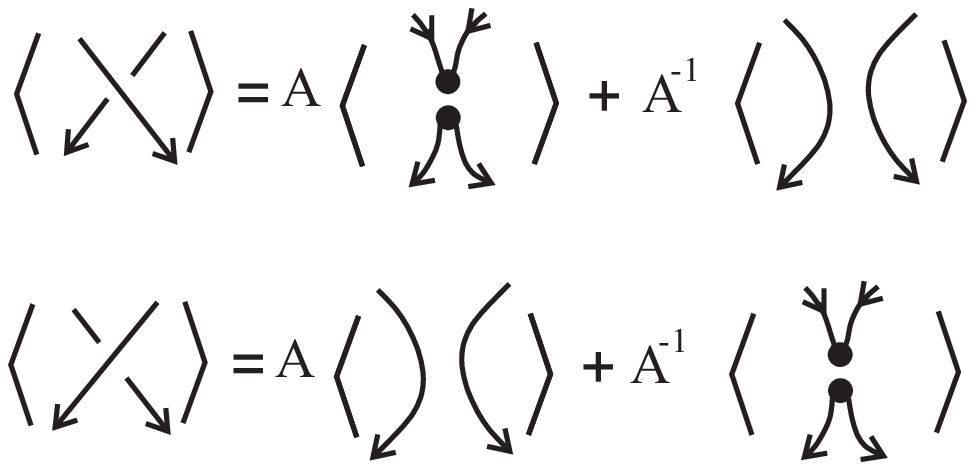}
    \renewcommand{\figurename}{Fig.}
  \caption{Oriented state expansion.}
  \label{Fig.arrow}
\end{figure}

\begin{figure}[!htbp]
  \centering
  \includegraphics[width=8cm]{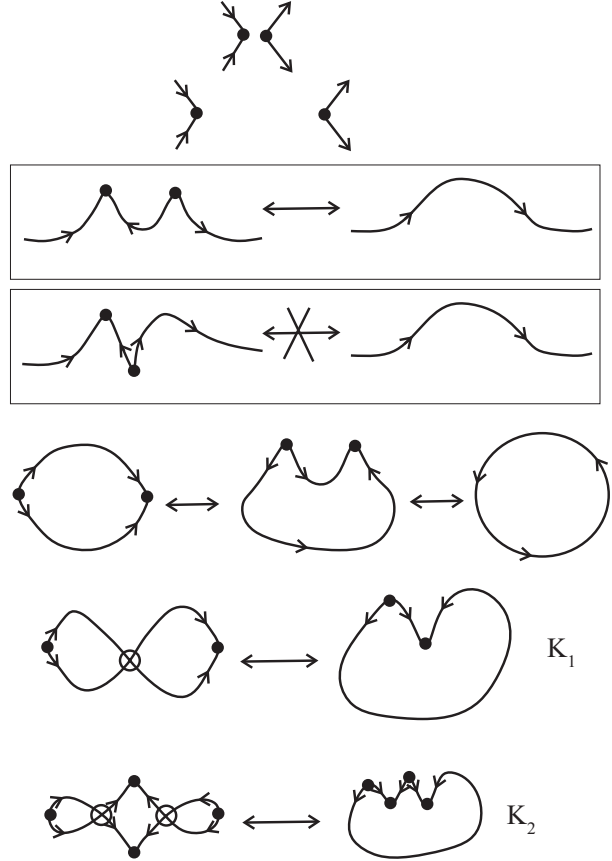}
    \renewcommand{\figurename}{Fig.}
  \caption{Reduction rule for the arrow polynomial.}
  \label{Fig.reduction}
\end{figure}

\begin{figure}[!htbp]
  \centering
  \includegraphics[width=12cm]{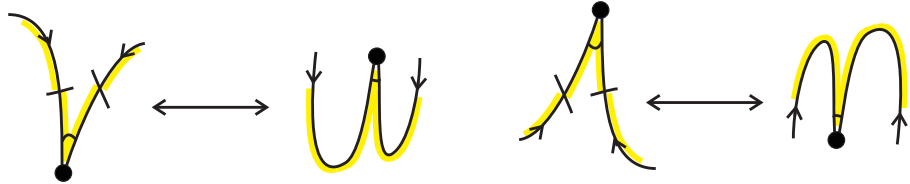}
    \renewcommand{\figurename}{Fig.}
  \caption{Reduction rule of a cusp with two bars for the arrow polynomial.}
  \label{Fig.cuspbar}
\end{figure}

\begin{figure}[!htbp]
  \centering
  \includegraphics[width=12cm]{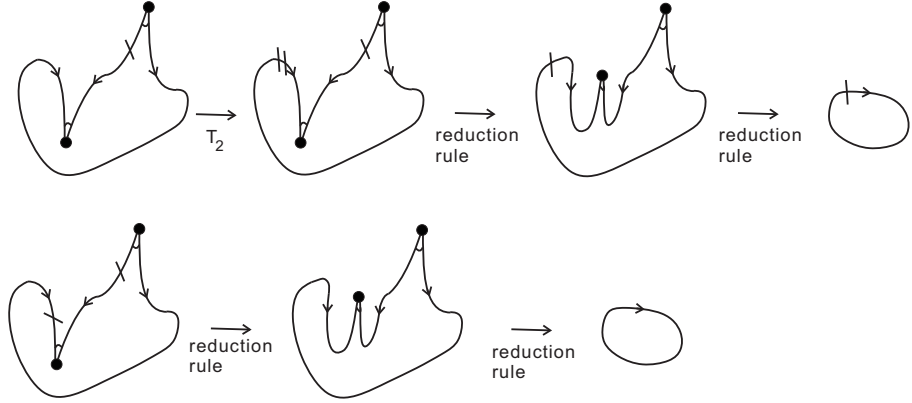}
    \renewcommand{\figurename}{Fig.}
  \caption{For the arrow polynomial, $K_1$ with one bar becomes a trivial loop with one bar by using T2 and reduction rule of a cusp with two bars, $K_1$ with two bars becomes a trivial loop by using T2 and reduction rule of a cusp with two bars,.}
  \label{Fig.K1bar}
\end{figure}

\begin{figure}[!htbp]
  \centering
  \includegraphics[width=8cm]{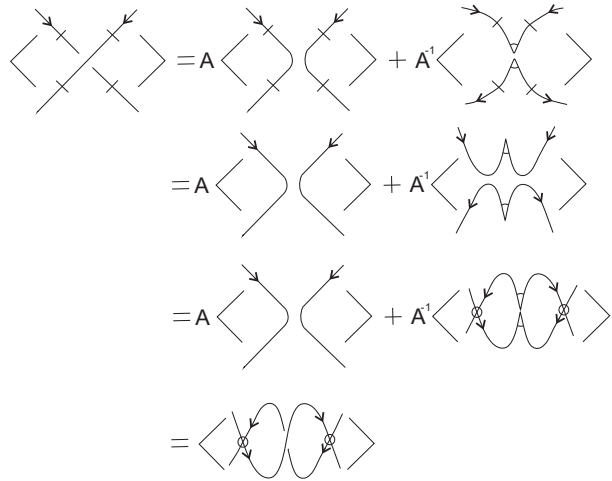}
    \renewcommand{\figurename}{Fig.}
  \caption{The invariance under move T3 for the arrow polynomial.}
  \label{Fig.T3cuspbar}
\end{figure}

\begin{definition}(\cite{DKArrow})
The arrow polynomial $\langle D\rangle_{A}(A,M)$ of an oriented twisted link diagram $D$ is defined by

\begin{equation}\label{arrowpoly}
  \langle D\rangle_{A}(A,M)=\sum_{S}A^{\alpha-\beta}d^{|S|-1}\langle S\rangle
\end{equation}

\noindent where $S$ runs over the oriented bracket states of the diagram, $\alpha$ denotes the number of smoothings
with coefficient $A$ in the state $S$ and $\beta$ denotes the number with coefficient
$A^{-1}$, $d=-A^{2}-A^{-2}$, $|S|$ is the number of circle graphs in the state, and
$\langle S\rangle$ is a product of extra variables $K_{1},K_{2},...$ and $M^{o(S)}$ associated with the non-trivial circle graphs where $o(S)$ is the number of trivial loops with a bar in the state $S$.
Note that each circle graph (trivial or not) contributes to the power of
$d$ in the state summation, but only non-trivial circle graphs contribute to
$\langle S\rangle$.
\end{definition}

In \cite{Bourgoin}, the writhe of twisted link diagram is a regular invariant of twisted link, that is, invariant under all extended Reidemeister moves except move RI.
Let $D$ be an oriented twisted link diagram with writhe $\omega(D)$.
The normalized version is defined by
\begin{equation}\label{Norarrowpoly}
  \langle D\rangle_{NA}(A,M)=(-A^{3})^{-\omega(D)}\langle D\rangle_{A}(A,M).
\end{equation}

For an oriented twisted link $L$ represented by an oriented link diagram $D$, we denote $\langle D\rangle_{NA}(A,M)$ by $\langle L\rangle_{NA}(A,M)$, and call it the \emph{normalized arrow polynomial} of $L$.

Note that there will not be bars in an oriented virtual link diagram, so does every state circle graph. Hence we will have the following conclusion as shown in \cite{DKArrow}.

\begin{thm}(\cite{DKArrow}, Theorem 1.6)\label{arrowinvar}
Let $L$ be an oriented virtual link.
Then $\langle L\rangle_{NA}(A,M)$ is
an invariant under the classical Reidemeister moves and virtual Reidemeister
moves.
\end{thm}

\begin{thm}\label{invofNA}
Let $L$ be an oriented twisted link.
Then $\langle L\rangle_{NA}(A,M)$ is
an invariant under the extended Reidemeister moves.
\end{thm}
\begin{proof}
It is sufficient to check the invariance under move T3 according to Theorem \ref{arrowinvar}.
It is easy to check that the arrow polynomial of twisted link diagram is invariant under move T3 as shown in Fig. \ref{Fig.T3cuspbar}.
\end{proof}

And it is obvious that we can get the following conclusion.
\begin{thm}
Let $D$ be an oriented twisted link diagram. The polynomial $\langle D\rangle_{A}(A,M)$ is an
invariant under the extended Reidemeister moves except the Reidemeister moves RI.
\end{thm}

By substituting $\frac{t^{\iota}+t^{-\iota}}{2}$ for $K_i$ and $(-A^{2}-A^{-2})^{-1}M$ for $M$, $\langle L\rangle_{NA}(A,M)$ turns into the polynomial $X_{D}(A,M,t)$ (or $Y_{D}(A,M,t)$) defined by Kamada in \cite{Kamada2010}.
By substituting $1$ for $K_i$ and $(-A^{2}-A^{-2})^{-1}M$ for $M$, $\langle L\rangle_{NA}(A,M)$ turns into the polynomial the twisted Jones polynomial defined by Bourgoin in \cite{Bourgoin}.

\begin{exam}
The twisted Jones polynomials of twisted links presented by the latter two diagrams in Fig. 1 are $A^{-6}(A^{2}-A^{4}+1)$. On the other hand, our normalized arrow polynomial invariants of them are $A^{-6}[A^{2}+(-A^{4}+1)K_1]$ and
$A^{-6}(A^{2}-A^{4}+1)$, respectively. (They are also distinguished by the twisted link groups \cite{Bourgoin}, which are the free
product of their upper and lower groups.)
\end{exam}

Note that if a reduced state has the form:
\begin{equation}\label{summand}
  A^{m}d^{l}(K^{j_{1}}_{i_{1}}K^{j_{2}}_{i_{2}}\cdots K^{j_{v}}_{i_{v}})M^{n}.
\end{equation}
Then the \emph{$k$-degree of the state} is:
\begin{equation}\label{kdeg}
 i_{1}\times j_{1} + i_{2}\times j_{2} +\cdots + i_{v}\times j_{v}
\end{equation}

\noindent which is equal to the half of reduced number of cusps in the state associated
with these variables.
A \emph{surviving state} is a summand of
$\langle D\rangle_{A}(A,M)$. The \emph{$k$-degree of a surviving state} is the $k$-degree of any state
associated with this summand.
Notice that if the summand has no $K_{C}$ variables, then
the $k$-degree is zero.

Let $AS(D)$ denote the set of $k$-degrees obtained from the set of surviving
states of a diagram $D$. The surviving states are represented by the summands
of $\langle D\rangle_{A}(A,M)$. That is, if $A^{3}K_{1}K_{4}M^{2}$ is a summand of $\langle D\rangle_{A}(A,M)$ then 5 is an element of $AS(D)$. If a twisted link $L$ has a total of 4 summands with subscripts summing
to: 2, 2, 1, 0 then $AS(L)=\{2, 1, 0\}$.

\begin{Lemma}(\cite{DKArrow}, Lemma 1.3)\label{ASD}
For an oriented virtual link diagram $D$, $AS(D)$ is invariant under the virtual and classical Reidemeister moves.
\end{Lemma}

Dye and Kauffman \cite{DKArrow} showed that the phenomenon of cusped states and extra
variables $K_{n}$ only occurs for virtual links.

\begin{thm}(\cite{DKArrow}, Theorem 1.5)
If $D$ be a classical link diagram then $AS(D)=\{0\}$.
\end{thm}

\begin{thm}
For an oriented twisted link diagram $D$,
$AS(D)$ is invariant under the extended Reidemeister moves.
\end{thm}
\begin{proof}
We conclude it by Theorem \ref{invofNA}.
\end{proof}

Next we shall determine the characteristics of arrow polynomials of checkerboard colorable twisted link diagrams. For this, we need to use the concept of twisted braid introduced by the author and the co-author in \cite{Xue}.
Let twisted braid $\mathcal{TB}_{n}$ be a $n$-strands braid generated by $\sigma_i$, $\upsilon_i$ and $b_i$ as shown in Fig.\ref{Fig.braidgenerators}.
For any $i,j=1,2,...,n-1$, the generator elements satisfy the following relations:
\begin{equation}\label{vrelation}
\begin{split}
  &\sigma_{i}\sigma_{i}^{-1}=1_{n} \\
&\sigma_{i}\sigma_{i+1}\sigma_{i}^{-1}=\sigma_{i+1}^{-1}\sigma_{i}\sigma_{i+1}\\
&\sigma_{i}\sigma_{j}=\sigma_{j}\sigma_{i} \ \ if  \ |i-j|\geq 2\\
       &b_ib_{i}=1_{n}\\
    &b_ib_{i+1}=b_{i+1}b_i\\
    &b_iv_{i}=v_{i}b_{i+1}\\
    &b_ib_{i+1}\sigma_{i}b_ib_{i+1}=v_{i}\sigma_{i}v_i\\
       &v_i^2=1_n\\
        &v_iv_j=v_jv_i \ \ if  \ |i-j|\geq 2\\
        &v_iv_{i+1}v_i=v_{i+1}v_iv_{i+1} \\
    &\sigma_{i}v_j=v_j\sigma_{i} \ \ if \ |i-j|\geq 2\\
    &v_i\sigma_{i+1}v_i=v_{i+1}\sigma_iv_{i+1} \\
  \end{split}
\end{equation}
\begin{figure}[!htbp]
  \centering
  \includegraphics[width=14cm]{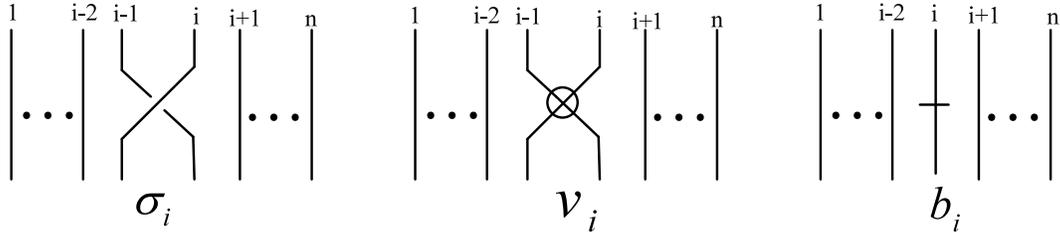}
    \renewcommand{\figurename}{Fig.}
  \caption{Generators $\sigma_i$, $\upsilon_i$ and $b_i$ of twisted braid $\mathcal{TB}_{n}$.}
  \label{Fig.braidgenerators}
\end{figure}

Kauffman and Lambropoulou \cite{KauLam} gave a general method for converting virtual links to
virtual braids. The braiding method given there is quite general and applies to
all the categories in which braiding can be accomplished. It includes the braiding of
classical, virtual, flat, welded, unrestricted, and singular knots and links.
For a twisted link diagram, we just explain the braiding techniques for a classical crossing, a virtual crossing and a bar by three schemas as shown in Fig. \ref{Fig.fulltwist}, \ref{Fig.halftwist} and \ref{Fig.movebar}, we refer the reader to \cite{KauLam,Xue}.

We give a special example for the braiding process of a twisted link as shown in Fig. \ref{Fig.example}.
We refer the reader to \cite{KauLam} and \cite{Xue} for explanation.
\begin{figure}[!htbp]
  \centering
  \includegraphics[width=14cm]{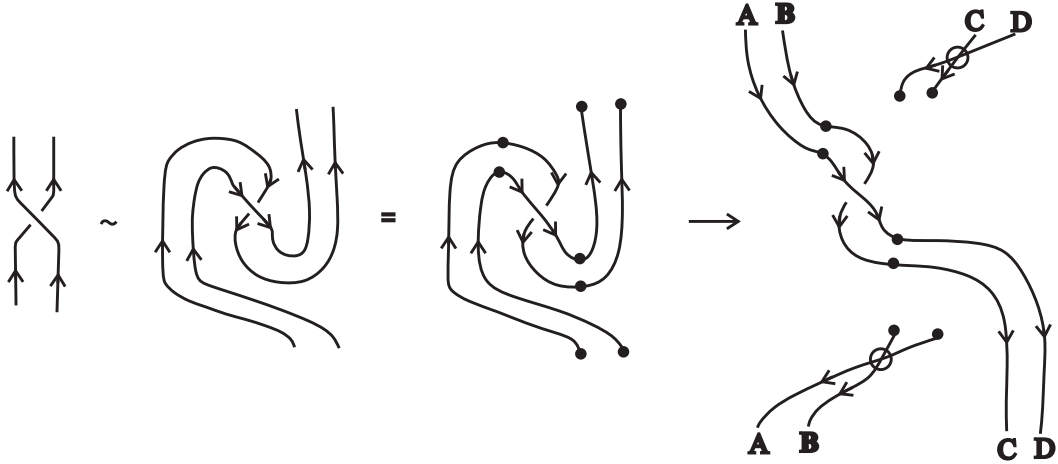}
    \renewcommand{\figurename}{Fig.}
  \caption{Full twist for a classical crossing.}
  \label{Fig.fulltwist}
\end{figure}

\begin{figure}[!htbp]
  \centering
  \includegraphics[width=14cm]{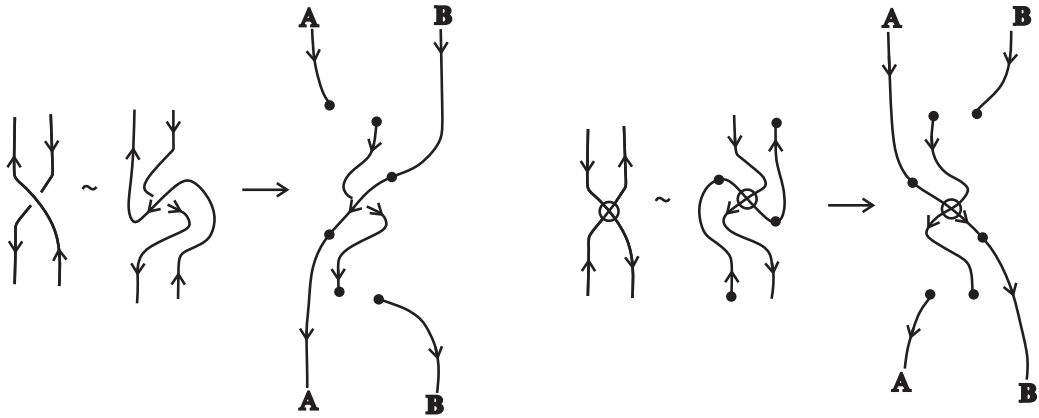}
    \renewcommand{\figurename}{Fig.}
  \caption{Half twist for a virtual crossing.}
  \label{Fig.halftwist}
\end{figure}

\begin{figure}[!htbp]
  \centering
  \includegraphics[width=14cm]{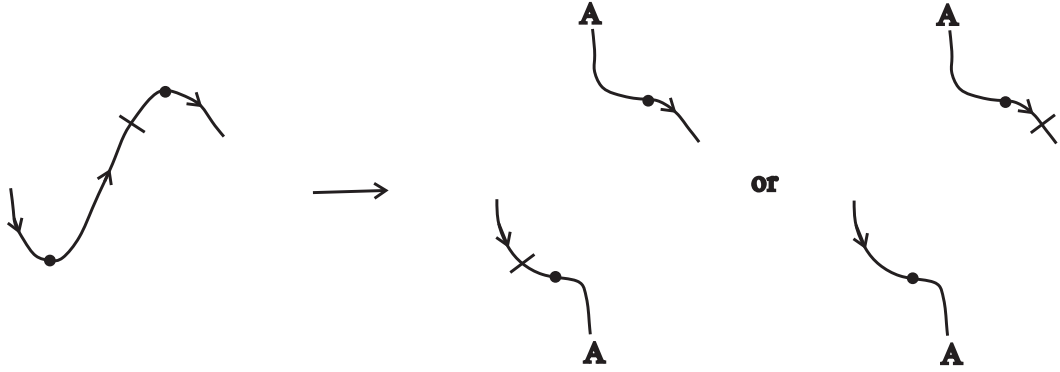}
    \renewcommand{\figurename}{Fig.}
  \caption{The braiding techniques for a bar.}
  \label{Fig.movebar}
\end{figure}

\vspace{5cm}
\begin{figure}[!htbp]
  \centering
  \includegraphics[width=14cm]{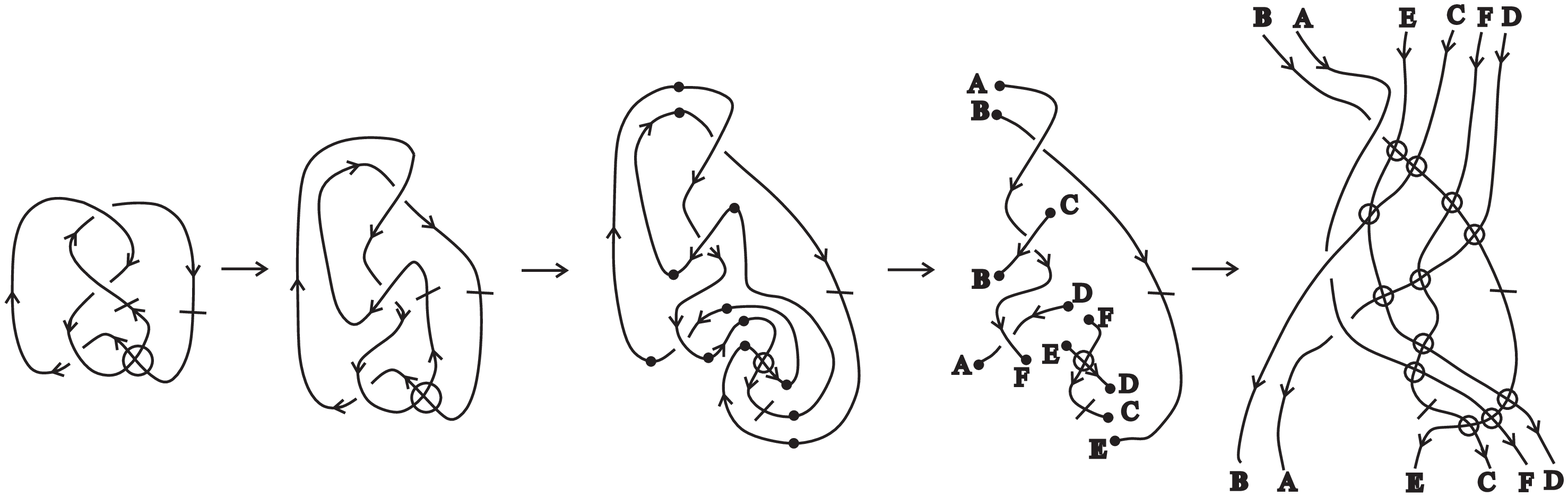}
    \renewcommand{\figurename}{Fig.}
  \caption{Twisted braid example.}
  \label{Fig.example}
\end{figure}

In Fig. \ref{Fig.example} we illustrate an example of braiding the arcs of the diagram.
In the first and second pictures we show a
twisted knot and its preparation for braiding by crossing rotation, respectively.
In the forth picture, we break each up-arc, and we name each pair of endpoints with the same letter.
Note that the closure of the twisted braid in Fig. \ref{Fig.example} is checkerboard colorable.

\begin{thm}(\cite{KauLam}, Theorem 1)\label{braiding}
Every (oriented) virtual link can be represented by a virtual
braid whose closure is isotopic to the original link.
\end{thm}

\begin{thm}(\cite{Xue})\label{twistedbraid}
Every (oriented) twisted link can be represented by a twisted
braid whose closure is isotopic to the original link.
\end{thm}
\begin{proof}
Based on the proof of Theorem 1 in \cite{KauLam}, we will adapt the similar technique for the case of twisted link.
So we leave the proof to the reader.
\end{proof}

\begin{remark}
According the braiding technique, described in Theorem 1 \cite{KauLam}, which just changes the relatively position of classical and virtual crossings (resp. bars) by crossing rotation (resp. by moving), thus the original twisted link diagram and the closure of its twisted braid have the same checkerboard colorability.
\end{remark}

We now determine the characteristics of arrow polynomials of checkerboard colorable twisted link diagrams.
We have the following statement that figures out the characteristics of cusped states and extra
variable $M$ for the arrow polynomials of checkerboard colorable twisted links.

\begin{thm}\label{main}
Let $D$ be an oriented checkerboard colorable twisted link diagram.
Then $\langle D\rangle_{NA}(A,M)$ has
\begin{enumerate}
  \item[(1)] $\langle D\rangle_{NA}(A,M)$ does not contain $M$;
  \item[(2)] $AS(D)$ only contains even integer;
  \item[(3)] If $\langle D\rangle_{NA}(A,M)$ contains a summand with $K_{i_{1}}^{j_{1}}K_{i_{2}}^{j_{2}}\cdots K_{i_{v}}^{j_{v}}$ ($i_1>i_2>\cdots i_v\geq1$), then $\sum_{t=1}^{v}i_t\cdot j_t$ is even and $i_1\leq \sum_{t=1}^{v}i_t\cdot j_t-i_1$.
\end{enumerate}
\end{thm}
\begin{proof}
Since $D$ is checkerboard colorable, then each circle graph with bars or not is checkerboard colorable.
Thus for each circle graph with bars will contain even bars.
According to the above statement in Remark \ref{importantstatement}, this kind of circle graph will reduce to $K_{1},K_{2},...$.
Thus Theorem \ref{main}(1) holds.

Next we shall explain that the Theorem \ref{main}(2) and \ref{main}(3) for the case of twisted link diagram can be reduced to the case of virtual link diagram in \cite{DJK2020} (Theorem 4.3(1) and (2)).
The key is to reduce all bars and have the same characteristics.

By Theorem \ref{twistedbraid}, we can arrange $D$ as twisted braid $\mathcal{B}$ with all strands oriented downwards whose closure $\mathcal{\overline{B}}$ is isotopic to $D$ and checkerboard colorable.
Note that $\mathcal{\overline{B}}$ and $D$ have the same set of classical crossings (see Fig. \ref{Fig.example}).
Assume that $C$ is one checkerboard coloring of $\mathcal{\overline{B}}$, in where we only draw one kind of color (yellow) to describe the coloring $C$.

Let $\sigma$ be an unreduced state of $\mathcal{\overline{B}}$ without applying reduction rule in Fig. \ref{Fig.reduction}.
According to the coloring $C$, we can obtain a checkerboard coloring $C^{\sigma}$ of state $\sigma$, in where, the colorings of a small neighbourhood of one side of each arcs are the same with the coloring $C$ of $\mathcal{\overline{B}}$ except small segments of arcs around some classical crossings (see Fig. \ref{Fig.cbs}).

\begin{figure}[!htbp]
  \centering
  \includegraphics[width=14cm]{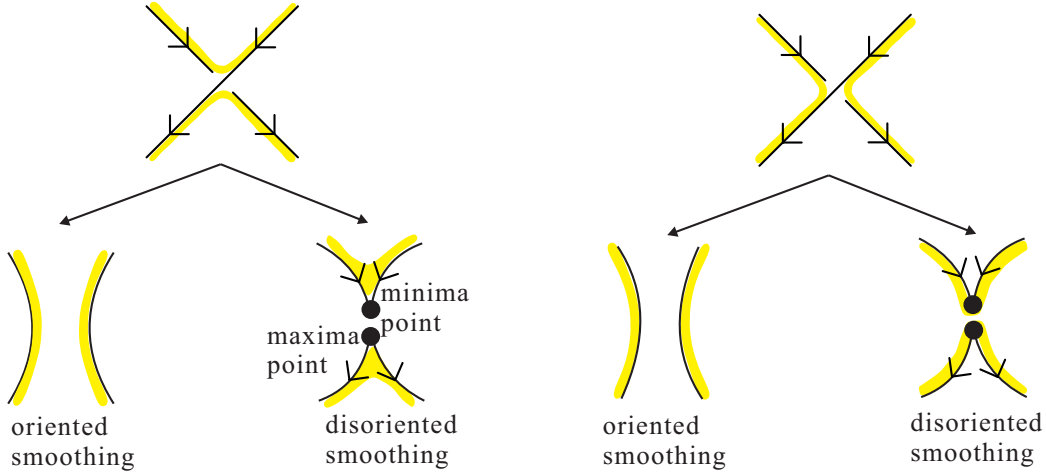}
    \renewcommand{\figurename}{Fig.}
  \caption{Checkerboard coloring $C^{\sigma}$ of state $\sigma$ (the crossing can be negative).}
  \label{Fig.cbs}
\end{figure}

On the one hand, we assume that the cusp $i$ will be read as $i^{<}$ (or $i^{>}$) if the acute (or obtuse) angle of cusp $i$ and color (yellow) are on the same side in a circle graph.
In all unreduced circle graphs of $\sigma$, there will be two $i^{<}$s (or $i^{>}$s) since the pair of acute angles of cusp $i$ and color (yellow) are on the same (or different) side.
A circle graph will have a \emph{corresponding word} if we read the cusp labeling $i$ $i^{<}$ (or $i^{>}$) of each classical crossing along a direction.
Note that we do not read out the labeling $b$ of bars, that is, there is not $b$ in the word since the information of bars is transferred to the word.
Note that the word is consistent with the virtual link diagram case in \cite{DJK2020} (see the proof of Theorem 4.3).

On the other hand, if we first reduce all bars in each circle graph $G$, then we will obtain a new circle graph $G'$ such that the placement relation between acute angle of every cusp and color does not change as shown in Fig. \ref{Fig.cuspbar}.
Thus the words of $G$ and $G'$ should be the same.

Hence, the remain proof is consistent with the virtual link diagram case in \cite{DJK2020} (see the proof of Theorem 4.3(1) and (2)).
\end{proof}

We can get the the same statement as Theorem 4 in \cite{Bourgoin}.
\begin{cor}
If a twisted link has a checkerboard colorable diagram, then its twisted Jones
polynomial is $(-A^{2}-A^{-2})$ times its Jones polynomial.
\end{cor}
\begin{proof}
Let $D$ be a checkerboard colorable diagram of checkerboard colorable twisted link $L$.
Since $\langle D\rangle_{NA}(A,M)$ does not contain $M$, and we let $K_i=1$ in $\langle D\rangle_{NA}(A,M)$, then its twisted Jones
polynomial $V_D(A,M)=(-A^{2}-A^{-2})\langle D\rangle_{NA}(A,(-A^{2}-A^{-2})^{-1}M)$ is just $(-A^{2}-A^{-2})$ times its Jones polynomial.
\end{proof}

Note that bar only deduce the cusps in circle graph.

\begin{Lemma}
Let $L$ be an oriented twisted link.
Then $b(L)$ is no less than the maximum degree of $M$ in $\langle L\rangle_{NA}(A,M)$.
\end{Lemma}

\begin{exam}
For virtual knot with 2 classical crossings, there exists 4 inequivalent twisted knots (not virtual knot) as shown in Fig. \ref{Fig.four2tl}.
Their arrow polynomials are as follows.
Thus the latter three twisted knots are non-checkerboard colorable by Theorem \ref{main}.

\begin{equation}
   \langle 2.1\rangle_{NA}=A^{-6}(A^{2}-A^{4}+1).
\end{equation}

\begin{equation}
   \langle 2.2 \rangle_{NA}=A^{-6}(((-1+A^4+A^6) M)/A^4).
\end{equation}

\begin{equation}
   \langle 2.3 \rangle_{NA}=A^{-6}(M \left(-\frac{\text{K1}}{A^4}+A^2-\text{K1}+2\right)).
\end{equation}

\begin{equation}
   \langle 2.4 \rangle_{NA}=A^{-6}(-\frac{M^2}{A^4}+A^2+\text{K1}-M^2+1).
\end{equation}

\begin{figure}[!htbp]
  \centering
  \includegraphics[width=8cm]{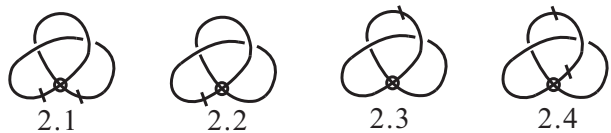}
    \renewcommand{\figurename}{Fig.}
  \caption{Four inequivalent twisted knots $2.1$, $2.2$, $2.3$, $2.4$.}
  \label{Fig.four2tl}
\end{figure}
\end{exam}

\section{Acknowledgements}
Deng is supported by Doctor's Funds of Xiangtan University (No. 09KZ$|$KZ08069) and NSFC (No. 12001464).
The project is also supported partially by Hu Xiang Gao Ceng Ci Ren Cai Ju Jiao Gong Cheng-Chuang Xin Ren Cai (No. 2019RS1057).

\end{document}